\documentclass[Roman,12pt,onecolumn,A4]{article}
\usepackage{amsmath}
\usepackage{amsthm}
\usepackage{amssymb}
\usepackage{amsfonts}
\usepackage{fancyhdr}
\usepackage{array}
\usepackage{pst-plot}
\usepackage{tikz}
\usepackage{tkz-euclide}
\usepackage[left=.5in,right=.5in,top=1in,bottom=.75in]{geometry}
\usepackage{enumerate}
\usepackage{ragged2e}
\usepackage{esint}
\usepackage{csvsimple}
\usepackage{cite}
\title{A Unified Approach to Computing the Zeros of Classical Orthogonal Polynomials}
\author{Ridha Moussa and James Tipton}
\newtheorem{thm}{Theorem}

\newtheorem{lem}[thm]{Lemma}

\newtheorem{prop}[thm]{Proposition}

\newcommand{\pd}[2]{\dfrac{\partial #1}{\partial #2}}

\begin{document}
\maketitle

\tableofcontents

\section{Introduction}

The Jacobi, Hermite, and Laguerre polynomials are collectively referred to as the classical orthogonal polynomials.  They have served as objects of study as early as the 19th century and have found applications to fields such as physics, approximation theory, and number theory.  The classical orthogonal polynomials may be characterized as solutions to a Sturm-Liouville type equation of the form
\[Q(x)y'' + L(x) y' + \lambda y = 0\]
In the case of the Jacobi polynomials, one has $Q(x) = 1- x^2$, $L(x) = \beta - \alpha -(\alpha + \beta + 2)x$, and $\lambda = n(n+\alpha + \beta + 1)$.  For the Hermite polynomials, take $Q(x) = 1$, $L(x) = -2x$, and $\lambda = $.  While the generalized Laguerre polynomials have $Q(x) = x$, $L(x) = (\alpha + 1 - x)$, and $\lambda = n$.\\

In each case, the corresponding polynomial solutions satisfy an orthogonality condition of the form
\[\int_{-\infty}^\infty P_m(x)P_n(x)W(x)dx = 0,\quad m\neq n\]
where $W(x) = \begin{cases}(1-x)^\alpha(1+x)^\beta, & -1\leq x\leq 1\\ 0, & |x|>1\end{cases}$ for the Jacobi polynomials, $W(x) = e^{-x^2}$ for the Hermite Polynomials, and $W(x) = \begin{cases}x^\alpha e^{-x}, & x\geq0\\ 0, & x<0\end{cases}$ for the generalized Laguerre polynomials.\\

There is much interest in the zeros of classical orthogonal polynomials, perhaps due in part to their well known electrostatic interpretation.  Most approaches to calculating these zeros are done so case by case using quadrature rules.  The approach we present here does not use quadrature, and to our knowledge does not appear in the literature.\\

In this note, we present a unified method to calculate the zeros of the classical orthogonal polynomials which is based in the electrostatic interpretation and its connection to the energy minimization problem. In section 2, we present the details of the method, while in section 3 we discuss the electrostatic interpretation in the context of the energy minimization problem.  In section 4 we provide some examples.  The paper is concluded with possible avenues of investigation.

\section{Method}

Given a polynomial $\displaystyle{y = c_n\prod_{i=1}^n(x-x_i)}$, where $c_n,x_i\in\mathbb{R}$, $c_n\neq 0$, and the $x_i$ are distinct, one has that:

\begin{align}
\dfrac{y'}{y}& =\sum_{i=1}^n\dfrac{1}{x-x_i}\\
\dfrac{y''}{y} & = \sum\limits_{i=1}^n\sum\limits_{j\in J_i}\dfrac{1}{(x-x_i)(x-x_j)} = 2\sum_{i<j}\dfrac{1}{(x-x_i)(x-x_j)}\\
\dfrac{ax^2+bx+c}{(x-x_i)(x-x_j)} & = a + \dfrac{ax_i^2+bx_i+c}{(x_i-x_j)(x-x_i)} + \dfrac{ax_j^2+b_j+c}{(x_j-x_i)(x-x_j)}
\end{align}
Identities (1) and (2) follow from the product rule.  Identity (3) follows from partial fraction decomposition.

\begin{lem}
In the above setting, one has that
\begin{align*}
    (\mu x+\nu)\dfrac{y'}{y} & = \mu n+\sum_{i=1}^n \dfrac{\nu+\mu x_i}{x-x_i}\qquad\text{and}\\
    (ax^2+bx+c)\dfrac{y^{''}}{y} & = a(n^2-n)+2\sum_{i\neq j}\dfrac{ax_i^2+bx_i+c}{(x_i-x_j)(x-x_i)}
\end{align*}

\end{lem}

\begin{proof}
The first identity follows directly from (1) and some long division.  For the second identity, combine (2) and (3) to get:
\[(ax^2+bx+c)\dfrac{y^{''}}{y} = 2\sum_{i<j}\left[a + \dfrac{ax_i^2+bx_i+c}{(x_i-x_j)(x-x_i)} + \dfrac{ax_j^2+bx_j+c}{(x_j-x_i)(x-x_j)}\right]\]
There are $\frac{n^2-n}{2}$ terms in the above summation, thus we get that $2\sum\limits_{i<j}a = a(n^2-n)$.  Now observe that
\begin{align*}
    \sum_{i<j}\left[ \dfrac{ax_i^2+bx_i+c}{(x_i-x_j)(x-x_i)} + \dfrac{ax_j^2+bx_j+c}{(x_j-x_i)(x-x_j)}\right]&  = \sum_{i<j} \dfrac{ax_i^2+bx_i+c}{(x_i-x_j)(x-x_i)} + \sum_{i<j}\dfrac{ax_j^2+bx_j+c}{(x_j-x_i)(x-x_j)}\\
    = \sum_{i<j} \dfrac{ax_i^2+bx_i+c}{(x_i-x_j)(x-x_i)} + \sum_{j<i}\dfrac{ax_i^2+bx_i+c}{(x_i-x_j)(x-x_i)}
   & = \sum_{i\neq j} \dfrac{ax_i^2+bx_i+c}{(x_i-x_j)(x-x_i)}
\end{align*}
where the second to last equality follows from index swapping on the second summation.  Putting the above calculations together yields the desired result.
\end{proof}

\begin{prop}
Suppose $y$ is a degree $n$ polynomial solution to the differential equation:
\[(ax^2 + bx + c)y^{''} + (\mu x + \nu)y^{'} + \kappa y = 0\]
If the zeros of $y$, $x_1$, \dots, $x_n$ are distinct, then for each integer $k\in[1,n]$, we have that
\[2\sum_{j\in J_k}\dfrac{ax_k^2+bx_k+c}{x_k-x_j}+\nu + \mu x_k=0\]
where $J_k$ consists of all integers in $[1,n]$ except $k$.
\end{prop}

\begin{proof}
Divide by $y$ and apply Lemma 1 to obtain:

\begin{align*}
     a(n^2-n)+2\sum_{i\neq j}\dfrac{ax_i^2+bx_i+c}{(x_i-x_j)(x-x_i)}+\mu n+\sum_{i=1}^n \dfrac{\nu+\mu x_i}{x-x_i} + \kappa = 0\\
     \iff 2\sum_{i\neq j}\dfrac{ax_i^2+bx_i+c}{(x_i-x_j)(x-x_i)}+\sum_{i=1}^n \dfrac{\nu+\mu x_i}{x-x_i} + M=0\\
     \iff 2(x-x_k)\sum_{i\neq j}\dfrac{ax_i^2+bx_i+c}{(x_i-x_j)(x-x_i)}+(x-x_k)\sum_{i=1}^n \dfrac{\nu+\mu x_i}{x-x_i} + (x-x_k)M=0
\end{align*}
where $M = \kappa + a(n^2-n)+\mu n$, and $k$ is some integer in $[1,n]$.  As $x$ approaches $x_k$, all terms will approach zero except those for which $i=k$.  Taking this limit gives the desired result.
\end{proof}

\section*{Jacobi Polynomials}

For $\alpha,\beta>-1$, the degree $n$ Jacobi polynomial $P_n^{(\alpha,\beta)}(x)$ solves the differential equation

\begin{align*}
(1-x^2)y^{''} + (\beta-\alpha-(\alpha + \beta+2)x)y' + n(n+\alpha+\beta+1)y=0
\end{align*}

Denote the $n$ distinct zeros of $P_n^{(\beta,\alpha)}(x)$ by $x_1$, \dots, $x_n$.  Let $a = -1$, $b = 0$, $c = 1$, $\mu = -(\alpha+\beta+2)$, and $\nu = \beta-\alpha$.  By Proposition $1$, we see that the zeros must satisfy
\begin{align*}
& 2\sum_{j\in J_k}\dfrac{-x_k^2+1}{x_k-x_j}+\beta-\alpha - (\alpha + \beta + 2) x_k=0\\
\iff & \dfrac{\frac{1}{2}(\alpha + 1)}{x_k-1} + \dfrac{\frac{1}{2}(\beta+1)}{x_k+1} + \sum_{j\in J_k}\dfrac{1}{x_k-x_j} = 0 \tag{4}
\end{align*}

In what follows, consider the real-valued function
\[f(\vec{x}) = \prod_{k=1}^n\Big[(1-x_k)^{(\alpha+1)/2}(1+x_k)^{(\beta+1)/2}\Big]\prod_{i<j}(x_j-x_i)\]
defined over the set $D_n = \{\vec{x}\in\mathbb{R}^n:-1<x_1<x_2<\cdots<x_n<1\}$.  Note that $f$ is smooth over $D_n$ and continuous on $\overline{D_n}$.  Note also that $f$ vanishes on the boundary of $D_n$, but is positive over $D_n$.  Since $f$ must attain an absolute maximum in $\overline{D_n}$, the previous observations show that this maximum occurs in $D_n$ and must be a critical point.

\begin{lem}
A point $\vec{x}\in D_n$ is a critical point of $f$
if and only if (1) holds for $k = 1,2,\dots, n$.
\end{lem}

\begin{proof}
Consider instead 
\[\ln(f) = \sum_{k=1}^n\Big[\frac{\alpha+1}{2}\ln(1-x_k) + \frac{\beta+1}{2}\ln(1+x_k)\Big] + \sum_{i<j}\ln(x_j-x_i)\]
we have that
\[\pd{ln(f)}{x_k} = \frac{f_{x_k}}{f} = \dfrac{\frac{1}{2}(\alpha + 1)}{x_k-1} + \dfrac{\frac{1}{2}(\beta+1)}{x_k+1} + \sum_{j\in J_k}\dfrac{1}{x_k-x_j}\]
demonstrating the claim.
\end{proof}

\begin{lem}
The function $\ln(f)$ has only one critical point in $D_n$.
\end{lem}

\begin{proof}
The claim holds if we can show that $\ln(f)$ is concave in $D_n$.  This in turn will follow if we can show that the Hessian of $\ln(f)$ is diagonally dominant and negative definite.  To that extent, observe that
\[\pd{^2\ln(f)}{x_k^2} = \dfrac{-\frac{1}{2}(\alpha + 1)}{(x_k-1)^2} - \dfrac{\frac{1}{2}(\beta+1)}{(x_k+1)^2} - \sum_{j\in J_k}\dfrac{1}{(x_k-x_j)^2}<0\]
and for $i\neq j$ that
\[\pd{^2\ln(f)}{x_ix_j} = -\sum_{j\in J_k}\dfrac{1}{(x_i-x_j)^2}<0\]
The Hessian is clearly diagonally dominant, and since the entries are all negative, it must also be negative definite.
\end{proof}

\section*{Hermite Polynomials}

The degree $n$ Hermite polynomial $H_n(x)$ solves the differential equation $y^{''}-2xy^{'}+2ny = 0$.  Denote the $n$ distinct zeros of $H_n(x)$ by $x_1$, \dots, $x_n$.  Let $a = b = \nu = 0$, $c = 1$, $\mu = -2$, and $\kappa = 2n$.  By Proposition 1, we see that the zeros must satisfy

\begin{align*}
& 2\sum_{j\in J_k}\dfrac{1}{x_k-x_j} - 2x_k = 0\\
\iff & \sum_{j\in J_k}\dfrac{1}{x_k-x_j} - x_k = 0 \tag{5}
\end{align*}

In what follows, consider the real-valued function
\[f(\vec{x}) = \prod_{i<j}\Big[x_j-x_i\Big]e^{-\frac{1}{2}\sum_{k=1}^nx_k^2}\]
defined over the set $D_n = \{\vec{x}\in\mathbb{R}^n:-\infty<x_1<x_2<\cdots<x_n<\infty\}$.  Note that $f$ is smooth, positive, and bounded over $D_n$, but approaches $0$ on the boundary.  Thus $f$ must have a critical point in $D_n$.

\begin{lem}
A point $\vec{x}\in D_n$  is a critical point of $f$ if and only if (2) holds for $k = 1,2,\dots, n$.
\end{lem}

\begin{proof}
Consider instead 
\[\ln(f) = \sum_{i<j}\ln(x_j-x_i) - \frac{1}{2}\sum_{k=1}^nx_k^2\]
we have that
\[\pd{ln(f)}{x_k} = \frac{f_{x_k}}{f} = \sum_{j\in J_k}\dfrac{1}{x_k-x_j} - x_k\]
demonstrating the claim.
\end{proof}

\begin{lem}
The function $\ln(f)$ has only one critical point in $D_n$.
\end{lem}

\begin{proof}
The claim holds if we can show that $\ln(f)$ is concave in $D_n$.  This in turn will follow if we can show that the Hessian of $\ln(f)$ is diagonally dominant and negative definite.  To that extent, observe that
\[\pd{^2\ln(f)}{x_k^2} = - \sum_{j\in J_k}\dfrac{1}{(x_k-x_j)^2}-1<0\]
and for $i\neq j$ that
\[\pd{^2\ln(f)}{x_ix_j} = -\sum_{j\in J_k}\dfrac{1}{(x_i-x_j)^2}<0\]
The Hessian is clearly diagonally dominant, and since the entries are all negative, it must also be negative definite.
\end{proof}

\section*{Laguerre Polynomials}

The degree $n$ generalized Laguerre polynomial $L_n^{(\alpha)}(x)$ solves the differential equation 
\[xy^{''}+(\alpha + 1 - x)y^{'} + ny = 0\]
Denote the $n$ distinct zeros of $L_n^{(\alpha)}(x)$ by $x_1$, \dots, $x_n$.  Let $a = c = 0$, $b=1$, $\mu = -1$, $\nu = \alpha + 1$, and $\kappa = n$.  By Proposition 1 we see that the zeros must satisfy

\begin{align*}
    2\sum_{j\in J_k}\dfrac{x_k}{x_k-x_j} + \alpha + 1 -x_k = 0\iff \sum_{j\in J_k}\dfrac{1}{x_k-x_j} + \dfrac{\frac{1}{2}(\alpha + 1)}{x_k} - \frac{1}{2} = 0 \tag{6}
\end{align*}

In what follows, consider the real-valued function
\[f(\vec{x}) = \prod_{i<j}\Big[x_j-x_i\Big]\prod_{k=1}^n\Big[x_k^{(\alpha+1)/2}\Big]e^{-\frac{1}{2}\sum_{k=1}^nx_k}\]
defined over the set $D_n = \{\vec{x}\in\mathbb{R}^n:0<x_1<x_2<\cdots<x_n<\infty\}$.  Note that $f$ is smooth, positive, and bounded over $D_n$, but approaches $0$ on the boundary.  Thus $f$ must have a critical point in $D_n$.

\begin{lem}
A point $\vec{x}\in D_n$  is a critical point of $f$ if and only if (3) holds for $k = 1,2,\dots, n$.
\end{lem}

\begin{proof}
Consider instead 
\[\ln(f) = \sum_{i<j}\ln(x_j-x_i) + \sum_{k=1}^n\Big[\frac{\alpha+1}{2}\ln x_k\Big] - \frac{1}{2}\sum_{k=1}^nx_k\]
we have that
\[\pd{ln(f)}{x_k} = \frac{f_{x_k}}{f} = \sum_{j\in J_k}\dfrac{1}{x_k-x_j} + \dfrac{\frac{1}{2}(\alpha + 1)}{x_k} - \frac{1}{2}\]
demonstrating the claim.
\end{proof}

\begin{lem}
The function $\ln(f)$ has only one critical point in $D_n$.
\end{lem}

\begin{proof}
The claim holds if we can show that $\ln(f)$ is concave in $D_n$.  This in turn will follow if we can show that the Hessian of $\ln(f)$ is diagonally dominant and negative definite.  To that extent, observe that
\[\pd{^2\ln(f)}{x_k^2} = - \sum_{j\in J_k}\dfrac{1}{(x_k-x_j)^2} - \dfrac{\frac{1}{2}(\alpha+1)}{x_k^2}<0\]
and for $i\neq j$ that
\[\pd{^2\ln(f)}{x_ix_j} = -\sum_{j\in J_k}\dfrac{1}{(x_i-x_j)^2}<0\]
The Hessian is clearly diagonally dominant, and since the entries are all negative, it must also be negative definite.
\end{proof}

\section{Electrostatic Interpretation and the Connection to the Energy Minimization Problem}

As detailed by Szego in \cite{Szegoe1975}, the zeros of the classical orthogonal polynomials may be interpreted as the equilibrium position to an electrostatic problem.  Stieltjes derived this connection in the case of the Jacobi polynomials in 1885.   In this case, the problem is to find the position of $n\geq 2$ unit ``masses' in the interval $[-1,1]$, given two fixed positive masses at $-1$ and $1$, for which electrostatic equilibrium is attained.\\

Interest in this connection has been steadily growing, see Marcell\`{a}n, Mart\'{i}nez-Finkelshtein, and Mart\'{i}nez-Gonz\'{a}lez \cite{MARCELLAN2007258} for details.  As noted in \cite{MARCELLAN2007258}, this is due in part to advances in the theory of logarithmic potentials as well as special functions from other areas of study, such as physics, combinatorics, and number theory.  In \cite{MARCELLAN2007258}, the authors consider the following natural questions:
\begin{enumerate}
    \item Can the electrostatic interpretation be generalized to other families of polynomials?
    \item Is it necessary to consider the global minimum of the energy; what about other equilibria?
\end{enumerate}

In regards to the first question, it is noted in \cite{MARCELLAN2007258} that Ismail (see \cite{ismail2000electrostatics}, \cite{ismail2000more}) has provided an electrostatic model for general orthogonal polynomials, in which the external field is given as the sum of a long range and short range potential.  For example, in \cite{ismail2000electrostatics}, an explicit formula is given for the total energy of the model at the equilibrium position, and this energy is shown to be minimum.  In the case of Freud weights, the total energy is shown to be asymptotic to $\frac{-n^2}{\alpha}\ln n$.\\ 

The authors of \cite{MARCELLAN2007258} consider a more general case where the weight function satisfies the Pearson equation, in particular with weight function corresponding to the Freud-type polynomials.  It is noted that in this case, the zeros of the Freud-Type polynomials provide a critical configuration for the total energy; but it is an open problem as to whether the zeros are in a stable equilibrium.  In regards to the second question, it is posited whether other types of equilibria are preserved in this case.\\

The authors of \cite{MARCELLAN2007258} also present a max-min characterization of the zeros of the Jacobi polynomials which is amenable to complex zeros of the family when the parameters fall out of the ``classical'' bounds.  Loosely speaking, the characterization shows that of all possible compact continua from -1 to 1 (within the complex plane), the energy (minimized over $n$ points for a given compact continua) is maximized over all compact continua when the $n$ points are the zeros of the Jacobi polynomial.\\

More recently, in regards to question 1 above, Ismail and Wang developed an electrostatic interpretation to quasi-orthogonal polynomials in \cite{ismail2019quasi}.  The main result is an analogue to one given in \cite{ismail2000electrostatics}.  In brief, it says that the equilibrium position of $n$ unit charges in the presence of a given external field is uniquely attained at the zeros of the associated quasi-orthogonal polynomials.

\section{Examples}

In the tables that follow, approximations of zeros are listed for a variety of classical orthogonal polynomials of a specified degree, $n$. The \textbf{Jacobi} column corresponds to the general Jacobi polynomial with $\alpha = \frac{1}{4}$ and $\beta = \frac{1}{8}$.  The \textbf{Chebyshev} column refers to the Chebyshev polynomials of the 1st kind, which correspond to Jacobi polynomials with $\alpha=\beta = \frac{-1}{2}$.  The \textbf{Gegenbauer} column corresponds to Jacobi polynomials with $\alpha = \beta = \frac{1}{4}$.  The \textbf{Legendre} column corresponds to Jacobi polynomials with $\alpha = \beta = 0$.  The \textbf{Laguerre} column corresponds to the classical Laguerre polynomials.  The \textbf{General Laguerre} column corresponds to Laguerre polynomials with $\alpha = 1$.\\

These results are obtained by using a straightforward implementation of Newton's method in the following way:  Let $n$ be a fixed natural number and consider the vector $\vec{x}=(x_1,x_2,\dots,x_n)$ which contains the zeros of the orthogonal polynomial of degree $n$ and $\vec{f} = (f_1,f_2,\dots,f_n)$ be a vector valued function.  With this notation, we can write the system of equations as $\vec{f}(\vec{x}) = \vec{0}$.  The nonlinear equation above is represented by (4) in the case of the Jacobi polynomials, by (5) in the case of the generalized Laguerre polynomials, and by (6) in the case of the Hermite polynomials.  As for the initial guess we relied on formulas given in Section 18.16 of \cite{NIST:DLMF}.\\

Since the exact roots are known for the Chebyshev case, one may calculate the exact error.  Using the infinity norm we have for $n = 20$ the exact error is $6.749\times10^{-17}$, while for $n = 25$ the exact error is $6.297\times10^{-17}$.  We also provide error estimates in each case using the infinity norm.

\begin{tabular}{c}
{\Large \textbf{Error Estimates: }$\mathbf{n = 20}$}\\
\begin{tabular}{|c|c|}
\hline
   \textbf{Polynomial}  & \textbf{Error Estimate} \\
   \hline
   Legendre  & $1.6064700823479085388\times10^{-16}$\\
   \hline
   General Jacobi $\alpha = 1/4$, $\beta = 1/8$ & $2.0443258006786251481\times10^{-16}$ \\
   \hline
   Gegenbauer & $2.4276213934271014550\times10^{-16}$\\
   \hline
   Chebyshev 1st Kind & $2.775557561562891350\times10^{-17}$ \\
   \hline
   Classical Laguerre & $7.0122389569333584353\times10^{-15}$ \\
   \hline
   General Laguerre $\alpha = 1$ & $1.0850726264919494635\times10^{-14}$\\
   \hline
   Hermite & $1.2572574676652352260\times10^{-16}$\\
   \hline
\end{tabular}
\end{tabular}

\noindent\\

\begin{tabular}{c}
{\Large \textbf{Error Estimates: }$\mathbf{n = 25}$}\\
\begin{tabular}{|c|c|}
\hline
   \textbf{Polynomial}  & \textbf{Error Estimate} \\
   \hline
   Legendre  & $2.1554887097997079110\times10^{-16}$\\
   \hline
   General Jacobi $\alpha = 1/4$, $\beta = 1/8$ & $1.1129640277756032144\times10^{-16}$ \\
   \hline
   Gegenbauer & $1.4290762156055028002\times10^{-16}$\\
   \hline
   Chebyshev 1st Kind & $1.3834655062070259971\times10^{-16}$ \\
   \hline
   Classical Laguerre & $8.9260826473499668326\times10^{-15}$ \\
   \hline
   General Laguerre $\alpha = 1$ & $2.4825341532472729961\times10^{-16}$\\
   \hline
   Hermite & $4.7043788112778503159\times10^{-16}$\\
   \hline
\end{tabular}
\end{tabular}

\noindent\\

\begin{tabular}{c}
{\large$\mathbf{n = 20}$ \textbf{with 30 iterations of Newton's Method}}\\
\\
\begin{filecontents*}{CSV/N20I30A.csv}
Jacobi ,Chebyshev,Gegenbauer
-0.992143445584654,-0.996917333733128,-0.991034230192877
-0.962098494639669,-0.972369920397677,-0.959770495283156
-0.90991914333223,-0.923879532511287,-0.906555627647643
-0.83679724371729,-0.852640164354092,-0.832601034386276
-0.744414638606914,-0.760405965600031,-0.739597864903566
-0.634897399553407,-0.649448048330184,-0.629673706991205
-0.510766182525352,-0.522498564715949,-0.505343420884813
-0.374878073128636,-0.382683432365090,-0.369451505240359
-0.230360787671044,-0.233445363855905,-0.22510699141448
-0.080540669675107,-0.0784590957278449,-0.0756123031135758
0.071133877871622,0.078459095727845,0.0756123031135758
0.221171767113119,0.233445363855905,0.22510699141448
0.366119581638305,0.382683432365090,0.369451505240359
0.502641066039214,0.522498564715949,0.505343420884813
0.627593920566186,0.649448048330184,0.629673706991205
0.738102136937797,0.760405965600031,0.739597864903566
0.831622222573934,0.852640164354092,0.832601034386276
0.906001841773546,0.923879532511287,0.906555627647643
0.959529848266796,0.972369920397677,0.959770495283156
0.99098031100982,0.996917333733128,0.991034230192877
\end{filecontents*}
\csvautotabular{CSV/N20I30A.csv}\end{tabular}

\begin{tabular}{c}
{\large$\mathbf{n = 20}$ \textbf{with 30 iterations of Newton's Method}}\\
\\
\begin{filecontents*}{CSV/N20I30B.csv}
Legendre,Laguerre,General Laguerre
-0.993128599185095,0.0705398896919887,0.174906752386615
-0.963971927277914,0.372126818001611,0.587303080638269
-0.912234428251326,0.916582102483273,1.23822510183424
-0.839116971822219,1.70730653102834,2.13139626007693
-0.746331906460151,2.74919925530943,3.27213313351699
-0.636053680726515,4.04892531385089,4.66749446588836
-0.510867001950827,5.61517497086162,6.32653619767384
-0.37370608871542,7.45901745367106,8.26067095201373
-0.227785851141645,9.5943928695811,10.4841673812082
-0.0765265211334974,12.0388025469643,13.0148487721526
0.0765265211334973,14.8142934426307,15.8750870127848
0.227785851141645,17.9488955205194,19.0932519076063
0.373706088715419,21.478788240285,22.7058938881731
0.510867001950827,25.4517027931869,26.7611702293794
0.636053680726515,29.9325546317006,31.3245161370075
0.746331906460151,35.013434240479,36.4887033461491
0.839116971822219,40.8330570567286,42.3934227457745
0.912234428251326,47.6199940473465,49.2688138498685
0.963971927277914,55.8107957500639,57.5544209713148
0.993128599185095,66.5244165256157,68.3770378145523
\end{filecontents*}
\csvautotabular{CSV/N20I30B.csv}
\end{tabular}

\newpage

\begin{tabular}{c}
{\large$\mathbf{n = 25}$ \textbf{with 30 iterations of Newton's Method}}\\
\\
\begin{filecontents*}{CSV/N25I30A.csv}
Jacobi,Chebyshev,Gegenbauer
-0.994901665878463,-0.998026728428272,-0.994174685362604
-0.975360959985654,-0.982287250728689,-0.973813483540093
-0.941256322689963,-0.951056516295154,-0.938979875687483
-0.893091307988287,-0.90482705246602,-0.890187770804335
-0.831584665110590,-0.844327925502015,-0.828161987824607
-0.757655035013272,-0.770513242775789,-0.75382448992158
-0.672406769138576,-0.684547105928689,-0.668280361715944
-0.577113343604359,-0.587785252292473,-0.57280131807384
-0.473198311079934,-0.481753674101715,-0.468806780981076
-0.362214026547642,-0.368124552684678,-0.357842771895352
-0.245818453819013,-0.248689887164855,-0.241558925568652
-0.125750396162197,-0.125333233564304,-0.121683964806954
-0.0038035200079399,8.36062906219094E-18,2.87922513006768E-17
0.118200440621912,0.125333233564304,0.121683964806954
0.238438898854630,0.248689887164855,0.241558925568652
0.355115642426439,0.368124552684678,0.357842771895352
0.466487667212620,0.481753674101715,0.468806780981076
0.570891216112889,0.587785252292473,0.57280131807384
0.666766634609609,0.684547105928689,0.668280361715944
0.752681672462637,0.770513242775789,0.75382448992158
0.827352885709386,0.844327925502015,0.828161987824607
0.889664827092574,0.904827052466020,0.890187770804335
0.938686772318027,0.951056516295154,0.938979875687483
0.973686941970036,0.982287250728689,0.973813483540093
0.994146438181037,0.998026728428272,0.994174685362604
\end{filecontents*}
\csvautotabular{CSV/N25I30A.csv}\end{tabular}

\newpage

\begin{tabular}{c}
{\large$\mathbf{n = 25}$ \textbf{with 30 iterations of Newton's Method}}\\
\\
\begin{filecontents*}{CSV/N25I30B.csv}
Legendre,Laguerre,General Laguerre
-0.995556969790498,0.0567047754527055,0.141236726258096
-0.976663921459518,0.299010898586989,0.473974537884425
-0.942974571228974,0.735909555435016,0.998383405621479
-0.894991997878275,1.36918311603519,1.71638168719236
-0.833442628760834,2.20132605372147,2.63069311458477
-0.759259263037358,3.23567580355804,3.7448777262027
-0.673566368473468,4.47649661507383,5.06340831233858
-0.577662930241223,5.92908376270045,6.59177560687321
-0.473002731445715,7.59989930995675,8.33662635980513
-0.361172305809388,9.49674922093243,10.3059430256137
-0.243866883720988,11.6290149117788,12.5092780113164
-0.12286469261071,14.0079579765451,14.9580612826525
-3.94351965660777E-18,16.6471255972888,17.6660089928416
0.12286469261071,19.5628980114691,20.6496747456588
0.243866883720988,22.775241986835,23.9292078044927
0.361172305809388,26.3087723909689,27.5294209021358
0.473002731445715,30.1942911633161,31.481337894211
0.577662930241223,34.471097571922,35.8245167628475
0.673566368473468,39.1906088039374,40.61069001566
0.759259263037358,44.422349336162,45.9097868582297
0.833442628760834,50.2645749938335,51.8206158754045
0.894991997878275,56.8649671739402,58.4916748142772
0.942974571228974,64.4666706159541,66.1674493598106
0.976663921459518,73.5342347921002,75.315081358106
0.995556969790498,85.260155562496,87.1338948199813
\end{filecontents*}
\csvautotabular{CSV/N25I30B.csv}
\end{tabular}

\newpage

\begin{tabular}{c}
{\large$\mathbf{n=25}$ \textbf{with 30 iterations of Newton's Method}}\\
\\
\begin{filecontents*}{CSV/HN25I30.csv}
Hermite
0.440147298645308
0.881982756213821
1.32728070207308
1.77800112433715
2.23642013026728
2.70532023717303
3.1882949244251
3.69028287699836
4.21860944438656
4.78532036735222
5.41363635528003
6.16427243405245
\end{filecontents*}
\csvautotabular{CSV/HN25I30.csv}

\end{tabular}

\section{Conclusion}

We have presented a unified approach for calculating the zeros of the classical orthogonal polynomials, and provided examples involving the Jacobi polynomials, including Chebyshev and Gengebauer, the General Laguerre polynomials, including Legendre and Laguerre, and the Hermite polynomials.  The approach has the potential to work for other cases of orthogonal polynomials, such as the Heine-Stietljes polynomials.  Future avenues of research include expanding the families of orthogonal polynomials this method applies.  There are also families for which very little is known about the zeros, such as the Generalized Bessel polynomials.

\newpage

\bibliographystyle{unsrt}
\bibliography{bibli}

\begin{thebibliography}{1}

\bibitem{Szegoe1975}
Gabor Szegö.
\newblock {\em Orthogonal Polynomials}.
\newblock American Mathematical Society, revised edition, 1975.

\bibitem{MARCELLAN2007258}
F.~Marcellán, A.~Martínez-Finkelshtein, and P.~Martínez-González.
\newblock Electrostatic models for zeros of polynomials: Old, new, and some
  open problems.
\newblock {\em Journal of Computational and Applied Mathematics}, 207(2):258 --
  272, 2007.
\newblock Proceedings of The Conference in Honour of Dr. Nico Temme on the
  Occasion of his 65th birthday.

\bibitem{ismail2000electrostatics}
Mourad~EH Ismail.
\newblock An electrostatics model for zeros of general orthogonal polynomials.
\newblock {\em Pacific journal of Mathematics}, 193(2):355--369, 2000.

\bibitem{ismail2000more}
Mourad~EH Ismail.
\newblock More on electrostatic models for zeros of orthagonal polynomials.
\newblock {\em Numerical functional analysis and optimization},
  21(1-2):191--204, 2000.

\bibitem{ismail2019quasi}
Mourad~EH Ismail and Xiang-Sheng Wang.
\newblock On quasi-orthogonal polynomials: Their differential equations,
  discriminants and electrostatics.
\newblock {\em Journal of Mathematical Analysis and Applications},
  474(2):1178--1197, 2019.

\bibitem{NIST:DLMF}
{\it NIST Digital Library of Mathematical Functions}.
\newblock http://dlmf.nist.gov/, Release 1.1.0 of 2020-12-15.
\newblock F.~W.~J. Olver, A.~B. {Olde Daalhuis}, D.~W. Lozier, B.~I. Schneider,
  R.~F. Boisvert, C.~W. Clark, B.~R. Miller, B.~V. Saunders, H.~S. Cohl, and
  M.~A. McClain, eds.

\end{thebibliography}

\end{document}